\documentclass[a4paper,11pt]{amsart}
\usepackage{hyperref,fancyhdr,mathrsfs,amsmath,amscd,amsthm,amsfonts,latexsym,amssymb,stmaryrd}
\usepackage[all]{xy}

\DeclareMathOperator{\rank}{rank}

\DeclareMathOperator{\dif}{d}

\newcommand{\Cal}{\mathcal{C}}

\def \phi{\varphi}
\def \Phi{\varPhi}
\def \p{\pi}

\def \s{\sigma}
\def \t{\tau}

\def \R{\mathbb{R}}
\def \Hq{\mathbb{H}\,}
\def \C{\mathbb{C}\,}

\def\widecheckg{g^{\hspace*{-2.5pt}\vbox to 5pt{\hbox to
0pt{\LARGE$\check{}$}}}\hspace*{2pt}}

\def\widecheckl{\lambda^{\hspace*{-3.5pt}\vbox to 8pt{\hbox to
0pt{\LARGE$\check{}$}}}\hspace*{2pt}}

\begin{document}

\title{A note on CR quaternionic maps}
\author{S.~Marchiafava and R.~Pantilie}
\thanks{S.M.\ acknowledges that this work was done under the program of GNSAGA-INDAM of C.N.R. and PRIN07 ``Geometria Riemanniana e strutture
differenziabili'' of MIUR (Italy).}
\email{\href{mailto:marchiaf@mat.uniroma1.it}{marchiaf@mat.uniroma1.it},
       \href{mailto:radu.pantilie@imar.ro}{radu.pantilie@imar.ro}}
\address{S.~Marchiafava, Dipartimento di Matematica, Istituto ``Guido~Castelnuovo'',
Universit\`a degli Studi di Roma ``La Sapienza'', Piazzale Aldo~Moro, 2 - I 00185 Roma - Italia}
\address{R.~Pantilie, Institutul de Matematic\u a ``Simion~Stoilow'' al Academiei Rom\^ane,
C.P. 1-764, 014700, Bucure\c sti, Rom\^ania}
\subjclass[2010]{Primary 53C26, Secondary 53C28}

\newtheorem{thm}{Theorem}[section]
\newtheorem{lem}[thm]{Lemma}
\newtheorem{cor}[thm]{Corollary}
\newtheorem{prop}[thm]{Proposition}

\theoremstyle{definition}

\newtheorem{defn}[thm]{Definition}
\newtheorem{rem}[thm]{Remark}
\newtheorem{exm}[thm]{Example}

\numberwithin{equation}{section}

\maketitle
\thispagestyle{empty}
\vspace{-10mm}
\section*{Abstract}
\begin{quote}
{\footnotesize
We introduce the notion of \emph{CR quaternionic map} and we prove that any such real-analytic map,
between CR quaternionic manifolds, is the restriction of a quaternionic map between quaternionic manifolds.
As an application, we prove, for example, that for any submanifold $M$, of dimension $4k-1$\,,
of a quaternionic manifold $N$, such that $TM$ generates a quaternionic subbundle of $TN|_M$,
of (real) rank $4k$\,, there exists, locally, a quaternionic submanifold of $N$, containing $M$ as a hypersurface.}
\end{quote}

\section*{Introduction}

\indent
An \emph{almost quaternionic manifold} is a manifold $M$ whose tangent bundle is a quaternionic vector bundle
(see, for example, \cite{AleMar-Annali96}\,, \cite{IMOP}\,, and, also, Section \ref{section:cr_q_mfds}\,, below).
Further, any compatible connection on $M$ induces an almost complex structure on the bundle $Z$ of admissible linear complex structures,
whose integrability characterises the integrability of the almost quaternionic structure of $M$ (that is, if $\dim M\geq 8$\,, 
the existence of a torsion free compatible connection; see \cite[Remark 2.10(2)]{IMOP}\,);
then $Z$ is the \emph{twistor space} of $M$.\\
\indent
Accordingly, a \emph{quaternionic map} \cite{IMOP} between quaternionic manifolds is, essentially, a map admitting a holomorphic lift
between the corresponding twistor spaces. This generalizes the well-known notion of quaternionic submanifold
(necessarily, totally geodesic with respect to any compatible torsion free connection \cite{AleMar-Report93}\,; see \cite{PePoSw-98}\,).
However, not all submanifolds of a quaternionic manifold are quaternionic (take, for example, any hypersurface).
Nevertheless, the generic submanifold, of codimension at most $2k-1$\,, of a quaternionic manifold, of dimension $4k$\,,
inherits a \emph{CR quaternionic structure}; moreover, any real-analytic CR quaternionic structure
is obtained this way through a germ unique embedding into a quaternionic manifold, which is called its \emph{heaven space}
\cite[Corollary 5.4]{fq}\,.\\
\indent 
In this note, we strengthen this result by proving that any real-analytic \emph{CR quaternionic map} (that is, any real-analytic map
admitting a CR lift between the corresponding twistor spaces) is the restriction of a quaternionic map between quaternionic
manifolds (Theorem \ref{thm:cr_q_maps}\,).
This shows that the CR quaternionic maps are the natural morphisms of the CR Quaternionic Geometry.\\
\indent 
We, also, apply this result to the study of (almost) CR quaternionic submanifolds, thus, obtaining Corollaries \ref{cor:cr_q_submfd_real_analytic} 
and \ref{cor:cr_q_submfd_codim}\,. By the latter, for any submanifold $M$, of dimension $4k-1$\,,
of a quaternionic manifold $N$, such that $TM$ generates a quaternionic subbundle of $TN|_M$,
of (real) rank $4k$\,, there exists, locally, a quaternionic submanifold of $N$, containing $M$ as a hypersurface.

\section{Brief review of CR quaternionic linear maps} \label{section:cr_q_lin_maps}

\indent
Recall (see \cite{IMOP}\,, \cite{fq} and the references therein) that a \emph{linear complex structure} on a (real) vector space $U$
is a linear map $J:U\to U$ such that $J^2=-{\rm Id}_U$\,. Then the $-{\rm i}$ eigenspace $C$ of $J$
satisfies $C\oplus\overline{C}=U^{\C}$; moreover, any such complex vector subspace of the complexification $U^{\C}$
is the $-{\rm i}$ eigenspace of a (unique) linear complex structure on $U$.\\
\indent
More generally, a \emph{linear CR structure} on a vector space $U$ is a complex vector subspace $C\subseteq U^{\C}$ such that
$C\cap\overline{C}=\{0\}$\,. A map $t:(U,C)\to (U',C')$ between CR vector spaces is a \emph{CR linear map} if it is linear and
$t(C)\subseteq C'$ (see \cite{fq} and \cite{fq_2} for these and, also, for the dual notion of \emph{linear co-CR structure}).\\
\indent
There are other ways to describe a linear CR structure. For example, if $C$ is a linear CR structure on $U$ and we denote
$E=U^{\C\!}/C$, and by $\iota:U\to E$ the composition of the inclusion $U\to U^{\C}$ followed by the projection $U^{\C\!}\to E$
then\\
\indent
\quad(a) $\iota$ is injective,\\
\indent
\quad(b) ${\rm im}\,\iota+J({\rm im}\,\iota)=E$\,,\\
\indent
\quad(c) $C=\iota^{-1}\bigl({\rm ker}(J+{\rm i})\bigr)$\,,\\
where $J$ is the linear complex structure of $E$. Moreover, the pair $(E,\iota)$ is unique, up to complex linear isomorphisms,
with these properties.\\
\indent
Thus, we could define a linear CR structure on a vector space $U$ as a pair $(E,\iota)$\,, where $E$ is a complex vector space
and $\iota:U\to E$ is a linear map satisfying properties (a) and (b)\,, above.\\
\indent
Furthermore, in this setting, the CR linear maps $t:(U,E,\iota)\to(U',E',\iota')$ are characterised by the fact that there exists
a (necessarily unique, if $t\neq0$) complex linear map $\widetilde{t}:E\to E'$ such that $\iota'\circ t=\widetilde{t}\circ\iota$\,.

\begin{lem} \label{lem:cr_subspace}
$t$ is injective if and only if\/ $\widetilde{t}$ is injective.
\end{lem}
\begin{proof}
The fact that $\widetilde{t}$ injective implies $t$ injective is obvious.
For the converse, it is sufficient to consider the case $U'=E'$.\\
\indent
Let $F$ be the complex vector subspace of $E'$ spanned by the image of $t$.
Then $\widetilde{t}$ decomposes into a complex linear isomorphism from $E$ onto $F$ followed by the
inclusion of $F$ into $E'$.
\end{proof}

\indent
Let $\Hq$ be the division algebra of quaternions, and note that its automorphism group is ${\rm SO}(3)$ acting trivially
on $\R$ and canonically on ${\rm Im}\Hq(=\R^3)$\,.\\
\indent
A \emph{linear quaternionic structure} on a vector space $U$ is an equivalence class of morphisms of associative algebras
from $\Hq\!$ to ${\rm End}(U)$\,, where two such morphisms $\s$ and $\t$ are equivalent if $\t=\s\circ a$\,, for some $a\in{\rm SO}(3)$\,.\\
\indent
Note that, if $\s:\Hq\!\to{\rm End}(U)$ defines a linear quaternionic structure on $U$ then the space
of \emph{admissible linear complex structures} $Z=\s(S^2)$ is well-defined \cite{AleMar-Annali96}\,.\\
\indent
Let $U$ and $U'$ be quaternionic vector spaces and let $Z$ and $Z'$ be the corresponding spaces of admissible linear complex structures,
respectively. A \emph{quaternionic linear map} from $U$ to $U'$ is a linear map $t:U\to U'$ for which there exists a function
$T:Z\to Z'$ such that $t\circ J=T(J)\circ t$\,, for any $J\in Z$\,; then, if $t\neq0$\,, we have that $T$ is unique and an
orientation preserving isometry \cite{IMOP}\,.

\begin{defn}[\,\cite{fq}\,]
1) A \emph{linear CR quaternionic structure} on a vector space $U$ is a pair $(E,\iota)$\,, where $E$ is a quaternionic vector space
and $\iota:U\to E$ is an injective linear map such that ${\rm im}\,\iota+J({\rm im}\,\iota)=E$\,, for any admissible linear
complex structure $J$ on $E$\,.\\
\indent
2) A \emph{CR quaternionic linear map} $t:(U,E,\iota)\to(U',E',\iota')$ between CR quaternionic vector spaces is a linear
map $t:U\to U'$ for which there exists a quaternionic linear map $\widetilde{t}:E\to E'$ such that $\iota'\circ t=\widetilde{t}\circ\iota$\,.
\end{defn}

\indent
Any quaternionic vector space of (real) dimension $4k$ is (non-canonically) isomorphic to $\Hq^{\!k}$.
More generally (and much less trivial), any CR quaternionic vector spaces is isomorphic to a finite product,
unique up to the order of factors, in which each factor is contained by one of the following two classes \cite[Corollary 3.7]{fq}\,.

\begin{exm}
For $k\geq1$\,, let $V_k$ be the vector subspace of $\Hq^{\!k}$ formed of all vectors of the form
$(z_1,\overline{z_1}+z_2\,{\rm j},z_3-\overline{z_2}\,{\rm j},\ldots)\,,$
where $z_1,\ldots,z_k$ are complex numbers with $\overline{z_k}=(-1)^kz_k$\,.\\
\indent
Then $(U_k,\Hq^{\!k})$ is a CR quaternionic vector space, where $U_k=V_k^{\perp}$ and we have used the canonical
Euclidean structure of $\Hq^{\!k}$.
\end{exm}

\begin{exm}
Let $V'_0=\{0\}(\subseteq\Hq\!)$ and, for $k\geq1$\,, let $V'_k$ be the vector subspace of $\Hq^{\!2k+1}$ formed of all vectors of the form
$(z_1,\overline{z_1}+z_2\,{\rm j},z_3-\overline{z_2}\,{\rm j},\ldots,\overline{z_{2k-1}}+z_{2k}\,{\rm j},-\overline{z_{2k}}\,{\rm j}\,)\,,$
where $z_1,\ldots,z_{2k}$ are complex numbers.\\
\indent
Then $(U'_k,\Hq^{\!2k+1})$ is a CR quaternionic vector space, where $U'_k=(V'_k)^{\perp}$.
\end{exm}

\indent
It follows that if $E$ is a quaternionic vector space, $\dim E=4k$\,, and $U\subseteq E$ is a generic vector subspace,
of codimension at most $2k-1$\,, then $(U,E)$ is a CR quaternionic vector space \cite{fq}\,.

\section{CR quaternionic manifolds} \label{section:cr_q_mfds}

\indent
An \emph{almost CR structure} on a manifold $M$ is a complex vector subbundle $\Cal$ of $T^{\C\!}M$ such that $\Cal\cap\overline{\Cal}=0$\,.
An \emph{(integrable almost) CR structure} is an almost CR structure whose space of sections is closed under the usual bracket.\\
\indent
If $M$ is a (real) hypersurface in a complex manifold $(N,J)$ then $T^{\C\!}M\cap{\rm ker}(J+{\rm i})$
is a CR structure on $M$. This admits a straightforward generalization to higher codimensions. Furthermore, any real-analytic
CR structure is obtained this way through a germ unique embedding into a complex manifold \cite{AndFre}\,.\\
\indent
Let $\mathfrak{a}$ be a (finite-dimensional) associative algebra. A \emph{bundle of associative algebras, with typical fibre $\mathfrak{a}$\,,}
is a vector bundle $A$, with typical fibre $\mathfrak{a}$\,, whose structural group is the automorphism group of $\mathfrak{a}$\,; in  particular,
each fibre of $A$ is an associative algebra (non-canonically) isomorphic to $\mathfrak{a}$\,. A \emph{morphism of bundles of associative algebras}
is a vector bundle morphism whose restriction to each fibre is a morphism of associative algebras.\\
\indent
A \emph{quaternionic vector bundle} is a vector bundle $E$ endowed with a pair $(A,\s)$\,, where $A$ is a bundle of associative algebras,
with typical fibre $\Hq$, and $\s:A\to{\rm End}(E)$ is a morphism of bundles of associative algebras.
Then the bundle $Z$ of \emph{admissible linear complex structures} on $E$ is the sphere bundle of $\s({\rm Im}\,A)$\,;
in particular, the typical fibre of $Z$ is the Riemann sphere.\\
\indent
An \emph{almost CR quaternionic structure} \cite{fq} on a manifold $M$ is a pair $(E,\iota)$\,, where $E$ is a quaternionic vector bundle over $M$
and $\iota:TM\to E$ is an injective vector bundle morphism such that $(E_x,\iota_x)$ is a linear CR quaternionic structure on $T_xM$,
for any $x\in M$ (see \cite{fq_2} for the dual notion of \emph{almost co-CR quaternionic structure}).\\
\indent
Let $(E,\iota)$ be an almost quaternionic structure on $M$ and let $\nabla$ be a compatible connection on $E$. Then $\nabla$ induces
a connection on the bundle $Z$ of admissible linear complex structures on $E$. Denote by $\mathcal{B}$ the
complex vector subbundle of $T^{\C\!}Z$ whose fibre, at each $J\in Z$, is the horizontal lift of
$\iota_{\p(J)}^{-1}\bigl({\rm ker}(J+{\rm i})\bigr)$\,, where $\p:Z\to M$ is the projection.
Then $\Cal=\mathcal{B}\oplus({\rm ker}\dif\!\p)^{0,1}$ is an almost CR structure on $Z$\,.\\
\indent
If $\Cal$ is integrable then $(E,\iota,\nabla)$ is a \emph{CR quaternionic structure}, $(M,E,\iota,\nabla)$ is a \emph{CR quaternionic manifold}
and $(Z,\Cal)$ is its \emph{twistor space}. Also, $(M,Z,\p,\Cal)$ is the \emph{twistorial structure} \cite{LouPan-II} of $(M,E,\iota,\nabla)$\,.\\ 
\indent 
Note that, if $\iota$ is a vector bundle isomorphism then we retrieve the usual notion of quaternionic structure/manifold
(see \cite[Remark 2.10(2)]{IMOP}\,).\\
\indent
Let $M$ be a hypersurface in a quaternionic manifold $N$. Then $(TN|_M,\iota,\nabla|_M)$ is a CR quaternionic structure on $M$,
where $\iota:TM\to TN|_M$ is the inclusion and $\nabla$ is any quaternionic connection on $N$ (that is, $\nabla$ is a torsion free
compatible connection on $N$). Further, this admits a straightforward generalization
to higher codimensions. Moreover, any real-analytic CR quaternionic structure is obtained this way through a germ unique embedding
into a quaternionic manifold, which is called its \emph{heaven space} \cite[Corollary 5.4]{fq}\,.
\section{CR quaternionic maps} \label{section:cr_q_maps}

\indent
Let $(M,E_M,\iota_M)$ and $(N,E_N,\iota_N)$ be almost CR quaternionic manifolds and let $\p_M:Z_M\to M$ and $\p_N:Z_N\to N$
be the corresponding bundles of admissible linear complex structures, respectively.
Also, let $\phi:M\to N$ and $\Phi:Z_M\to Z_N$ be maps.\\
\indent
We say that $\phi:(M,E_M,\iota_M)\to(N,E_N,\iota_N)$ is an \emph{almost CR quaternionic map, with respect to $\Phi$}, if
$\dif\!\phi_x$ is CR quaternionic, with respect to $\Phi_x$\,, at each point $x\in M$.\\
\indent
Suppose that $E_M$ and $E_N$ are endowed with compatible connections $\nabla^M$ and $\nabla^N$, respectively, with respect
to which $M$ and $N$ become CR quaternionic manifolds; denote by $\Cal^M$ and $\Cal^N$ the corresponding CR structures on $Z_M$ and
$Z_N$\,, respectively.\\
\indent
The map $\phi:(M,E_M,\iota_M)\to(N,E_N,\iota_N)$ is \emph{CR quaternionic, with respect to $\Phi$}, if
$\p_N\circ\Phi=\phi\circ\p_M$ and $\Phi:(Z_M,\Cal^M)\to(Z_N,\Cal^N)$ is a CR map.\\
\indent
The notions of \emph{(almost) CR quaternionic immersion/submersion/diffeomorphism} are defined, accordingly.\\
\indent
Note that, the CR quaternionic maps are just twistorial maps (see \cite{LouPan-II} for the definition of the latter).\\
\indent
Also, for maps of rank at least one, between quaternionic manifolds, the two notions coincide \cite[Theorem 3.5]{IMOP}\,.
That is why we call \emph{quaternionic maps} the (almost) CR quaternionic maps, between (almost) quaternionic manifolds.\\
\indent
Furthermore, any CR quaternionic map is, obviously, almost CR quaternionic. However, the converse does not hold,
as the following example shows.

\begin{exm} \label{exm:3d}
Let $(M,c)$ be a three-dimensional conformal manifold and let $L=(\Lambda^3TM)^{1/3}$ be the line bundle of $M$.
Then $E=L\oplus TM$ is oriented and $c$ induces on it a linear conformal structure. Therefore $E$ is a quaternionic vector
bundle and $(M,E,\iota)$ is an almost CR quaternionic manifold, where $\iota:TM\to E$ is the inclusion.\\
\indent
Furthermore, if $D$ is a conformal connection on $(M,c)$ and $\nabla=D^L\oplus D$, where $D^L$ is the connection induced by $D$ on $L$,
then $(M,E,\iota,\nabla)$ is a CR quaternionic manifold (this is a straightforward consequence of \cite[Theorem 4.6]{fq}\,).\\
\indent
Let $D'$ be another conformal connection on $(M,c)$ and let $\nabla'=(D')^L\oplus D'$, where $(D')^L$ is the connection induced by $D'$ on $L$.\\
\indent
Then ${\rm Id}_M:(M,E,\iota,\nabla)\to(M,E,\iota,\nabla')$ is a CR quaternionic map if and only if
the trace-free self-adjoint part of $*T$ is zero, where $*$ is the Hodge star-operator of $(M,c)$ and
$T$ is the difference between the torsion tensors of $D$ and $D'$.\\
\indent
As any conformal connection is determined by its torsion and the connection it induces on the line bundle of the manifold,
we can therefore easily construct many almost CR quaternionic maps which are not CR quaternionic.\\
\indent
A straightforward generalization shows that the same applies to any almost CR quaternionic manifold $(M,E,\iota)$
with $\dim M=2k+1$ and $\rank E=4k$\,, for some nonzero natural number $k$ (this follows from the fact that then
for any compatible connection $\nabla$ on $E$ we have that $(M,E,\iota,\nabla)$ is a CR quaternionic manifold).
\end{exm}

\begin{thm} \label{thm:cr_q_maps}
Let $M$ and $N$ be real-analytic CR quaternionic manifolds and let $\widetilde{M}$ and $\widetilde{N}$ be the corresponding
heaven spaces, respectively.\\
\indent
If $\phi:M\to N$ is a real-analytic map whose differential is nowhere zero then the following assertions are equivalent.\\
\indent
\quad{\rm (i)} $\phi$ is CR quaternionic (with respect to some real-analytic lift between the twistor spaces of
$M$ and $N$).\\
\indent
\quad{\rm (ii)} $\phi$ is the restriction of a germ unique quaternionic map $\widetilde{\phi}$ defined on some
open neighbourhood of $M$ in $\widetilde{M}$ and with values in $\widetilde{N}$. 
\end{thm}
\begin{proof}
Let $Z_{\widetilde{M}}$ and $Z_{\widetilde{N}}$ be the twistor spaces of $\widetilde{M}$ and $\widetilde{N}$, respectively.
If we denote by $Z_M$ and $Z_N$ their restrictions to $M$ and $N$ we obtain (generic) CR submanifolds
which are the twistor spaces of $M$ and $N$, respectively. Denote by $\p_{\widetilde{M}}:Z_{\widetilde{M}}\to\widetilde{M}$
and $\p_{\widetilde{N}}:Z_{\widetilde{N}}\to\widetilde{N}$ the projections and by $\p_M$ and $\p_N$ their
restrictions to $Z_M$ and $Z_N$\,, respectively.\\
\indent
Suppose that (i) holds and let $\Phi:Z_M\to Z_N$ be a real-analytic CR map with respect to which $\phi$ is CR quaternionic.
Note that, locally, any real-analytic CR function on $Z_M$ is the restriction of a unique
holomorphic function on $Z_{\widetilde{M}}$\,. Thus, if $f$ is a holomorphic function, locally defined on $Z_{\widetilde{N}}$\,,
then $f\circ\Phi$ is the restriction of some holomorphic function, locally defined on $Z_{\widetilde{N}}$\,. Consequently,
$\Phi$ is the restriction of some germ unique holomorphic map $\widetilde{\Phi}$ defined on some open neighbourhood $Z$
of $Z_M$ in $Z_{\widetilde{M}}$\,.\\
\indent
As $\Phi$ restricted to each fibre of $\p_M$ is a holomorphic diffeomorphism (in fact, an orientation preserving isometry),
by passing if necessary to an open subset of $Z$, we may suppose that $\widetilde{\Phi}$ restricted to each twistor line
is a holomorphic diffeomorphism. Then $\widetilde{\Phi}$ maps the family of twistor lines contained by $Z$
into a family of complex projective lines which, obviously, contains the fibres of $\p_N$\,.
As the families of twistor lines (on the twistor spaces of quaternionic manifolds) are locally complete,
we obtain that $\widetilde{\Phi}$ maps twistor lines to twistor lines.\\
\indent
Let $\t_{\widetilde{M}}$ and $\t_{\widetilde{N}}$ be the conjugations on $Z_{\widetilde{M}}$ and $Z_{\widetilde{N}}$\,, respectively,
which on the fibres of the projections $\p_{\widetilde{M}}$ and $\p_{\widetilde{N}}$ are given by the antipodal map.
Obviously, $\t_{\widetilde{N}}\circ\widetilde{\Phi}\circ\t_{\widetilde{M}}:\t_{\widetilde{M}}(Z)\to Z_{\widetilde{N}}$
is, also, a holomorphic extension of $\Phi$. Hence, $\t_{\widetilde{N}}\circ\widetilde{\Phi}={\Phi}\circ\t_{\widetilde{M}}$
on $Z\cap\tau_{\widetilde{M}}(Z)$ (which, obviously, contains $Z_M$\,).\\
\indent
Therefore $\widetilde{\Phi}$ maps real twistor lines to real twistor lines. But the real twistor lines are just
the fibres of $\p_{\widetilde{M}}$ and $\p_{\widetilde{N}}$\,. Hence, $\widetilde{\Phi}$ descends to a quaternionic
map, as required.\\
\indent
Suppose, now, that (ii) holds and let $\widetilde{\phi}:\widetilde{M}\to\widetilde{N}$ be a quaternionic extension of~$\phi$\,.
Then by \cite[Theorem 3.5]{IMOP} we have that $\widetilde{\phi}$ admits a holomorphic lift
$\widetilde{\Phi}:Z_{\widetilde{M}}\to Z_{\widetilde{N}}$ which, obviously, restricts to a CR map $\Phi:Z_M\to Z_N$
which is a lift of $\phi$\,. Thus, $\phi$ is CR quaternionic. 
\end{proof}

\indent 
For applications we shall, also, need the following:  

\begin{lem} \label{lem:cr_q_submfd}
Let $M$ be an almost CR quaternionic submanifold of a CR quaternionic manifold $N$; denote by $E$ the corresponding 
quaternionic vector bundle on $M$.\\ 
\indent 
Then there exists a connection on $E$ with respect to which $M$ is a CR quaternionic submanifold of $N$. 
\end{lem} 
\begin{proof} 
Let $Z$ be the bundle of admissible linear complex structures on $E$. Then the connection on the corresponding quaternionic 
vector bundle over $N$ induces a connection on $Z$\,.\\ 
\indent 
Locally, we can write $E^{\C\!}=H\otimes F$, with $H$ and $F$ complex vector bundles of ranks $2$ and $2k$\,, 
and structural groups ${\rm Sp}(1)$ and ${\rm GL}(k,\Hq)$\,, respectively, where $\rank E=4k$\,.\\ 
\indent 
Then $Z=PH$ and the connection on $Z$ corresponds to a connection $\nabla^{H}$ on $H$. Thus, if $\nabla^F$ is any connection 
on $F$ (compatible with its structural group) then $\nabla^H\otimes\nabla^F$ is the complexification of a connection on $E$, which is as required. 
\end{proof} 

\indent 
More can be said about the CR quaternionic submanifolds of a quaternionic manifold. 

\begin{cor} \label{cor:cr_q_submfd_real_analytic}  
Let $(M,E)$ be a real-analytic almost CR quaternionic submanifold of a quaternionic manifold $N$. Then the following 
assertions hold:\\ 
\indent 
{\rm (i)} The inclusion from $M$ into $N$ extends to a germ unique quaternionic immersion from a quaternionic manifold $P$ to $N$, 
with $\dim P=\rank E$.\\ 
\indent 
{\rm (ii)} Any quaternionic connection on $N$ induces, by restriction, a connection on $E$ with respect to which 
$M$ is a CR quaternionic submanifold of $N$.
\end{cor}
\begin{proof} 
Assertion (i) is an immediate consequence of Theorem \ref{thm:cr_q_maps}\,, and Lemmas \ref{lem:cr_subspace} and \ref{lem:cr_q_submfd}\,. 
Assertion (ii) is an immediate consequence of (i) and the fact that the quaternionic submanifolds are totally geodesic with respect to any quaternionic connection 
on the ambient space \cite{AleMar-Report93} (see \cite{IMOP}\,). 
\end{proof}

\indent 
We end with the following result (cf.\ \cite[Proposizione 12.1]{Mar-Rend90}\,). 

\begin{cor} \label{cor:cr_q_submfd_codim}
Let $\phi:M\to N$ be an immersion from a manifold of dimension $4k-1$ to a quaternionic manifold $N$ of dimension $4n$\,.
Suppose that, for any $x\in M$, the quaternionic vector subspace of $T_{\phi(x)}N$ spanned by $\dif\!\phi(T_xM)$ has dimension $4k$\,.\\
\indent
Then $\phi$ extends to a germ unique quaternionic immersion from a quaternionic manifold, of dimension $4k$\,, to $N$;  
in particular, the pull-back of any quaternionic connection on $N$ preserves the quaternionic vector bundle generated by $TM$.
\end{cor} 
\begin{proof} 
Let $Z_M$ and $Z_N$ be the twistor spaces of $M$ and $N$, respectively. By \cite[Proposition 1.10]{AndFre}\,, in an open neighbourhood $U$ of each point 
of $Z_M$ there exists a complex manifold $Z^U$, with $\dim_{\C\!}(Z^U)=2k+1$, containing $U$ as a CR submanifold. Moreover, 
we may suppose that any CR function on $U$ can be locally (and uniquely) extended to a holomorphic function on $Z^U$ (use \cite[Proposition III.2.3]{KoNo} 
to show that \cite[Theorem 14.1.1]{Bog} can be applied; cf.\ \cite{LeB-CR_twistors}\,). But this is sufficient for the proof of the global embeddability 
theorem \cite{AndFre} to work, thus, showing that there exists a complex manifold $Z$, with $\dim_{\C\!}Z=2k+1$, which contain $Z_M$ as an embedded CR submanifold. 
Consequently, also Theorem \ref{thm:cr_q_maps} extends to this setting. 
\end{proof}

\end{document}